\documentclass[12pt]{amsart}
\usepackage{amsmath,amsthm,epsfig,amssymb,amscd,amsfonts} 
\usepackage[all]{xy}
\usepackage{young}
\newtheorem{thm}[subsection]{Theorem}

\newtheorem{prop}[subsection]{Proposition}

\newtheorem{lem}[subsection]{Lemma}

\newtheorem{rem}[subsection]{Remark}
\theoremstyle{definition}
\newtheorem{Def}[subsection]{Definition}
\newtheorem{Not}[subsection]{Notation}

\newtheorem{proposition-definition}[subsection]{Proposition-Definition}

\begin{normalsize} \end{normalsize}

\newcommand{\ZZ}{{\mathbb Z}}
\newcommand{\QQ}{{\mathbb Q}}
\newcommand{\PP}{{\mathbb P}}
\newcommand{\NN}{{\mathbb N}}

\newcommand{\OOO}{{\mathcal O}}

\newcommand{\LLL}{{\mathcal L}}

\newcommand{\SSS}{{\mathcal S}}

\numberwithin{equation}{section}

\author{F. Laytimi}

\address {F. L.: Math\'ematiques - b\^{a}t. M2, Universit\'e Lille 1,
F-59655 Villeneuve d'Ascq Cedex, France}

\email {fatima.laytimi@math.univ-lille1.fr}

\author{W. Nahm}

\address{W. N.: Dublin Institute for Advanced Studies,
10 Burlington Road, Dublin 4, Ireland}

\email{wnahm@stp.dias.ie}


\title{  On $k$-ampleness equivalence}

\begin{document}

\date{}

\begin{abstract} 

For a partition $a$ and a vector bundle $E$ on a projective variety $X$ let $\mathcal{F}l_s(E)$ be the corresponding flag manifold. There is a line bundle $\it Q_a^s$ on $\mathcal{F}l_s(E)$ with
 $p:\mathcal{F}l_s(E)\to X $ and $\it p_*Q_a^s= \SSS_aE $ 
 We prove, if  
$ \SSS_aE $ is  $k$-ample (in the sense of Sommese) then  $\it Q_a^s$ is $k$-ample. For the inverse if  $\it Q_a^s$ is $k$-ample,  we prove that one of two the conditions of k-ampleness namely the cohomological vanishing is proved here  but not yet the condition of  semiamplenes of $\SSS_aE $ .
 
\end{abstract}

\maketitle

\section{Introduction} \setcounter{page}{1}
Let  $E$ be  a complex vector bundle  of rank $d $  on a projective manifold $X$
and $s$ be a sequence of integers $s=(s_0,s_1,\ldots s_m)$ such that  $0=s_0<s_1<\ldots<s_m=d.$ Over $X$ consider 
the manifold $\mathcal {F}l_s(E)$ of incomplete flags given by nested subspaces $V_i$ of codimensions $s_i$ of the 
fibers $V$ of $E$. 
$$V=V_{s_0}\supset V_{s_1}\supset\ \ldots \supset V_{s_m}=\{0\}.$$

Let $a=(a_1,\ldots,a_d)$ be a partition such that  for  $0<j<d$
one has $a_j>a_{j+1}$ for $j\in s$ and $a_j=a_{j+1}$ otherwise.
There is a corresponding line bundle $\it Q_a^s$ over 
$\mathcal {F}l_s(E),$ locally given by 
 $$ \it Q_a^s= \it  Q^{a_{s_1}}_1\otimes \it Q^{a_{s_2}}_2\otimes \ldots \otimes \it Q^{a_{s_m}}_m,$$
where
$$ \it {Q_j}=\det(V_{j-1}/V_j),\ \  1\leq j\leq m.$$
When restricted to a fiber $\mathcal {F}l_s(V)$, the line bundle $\it Q_a^s$ is very ample.
Let $\pi: \mathcal {F}l_s(E)\rightarrow X$ be the natural projection. One has
\begin{equation} \pi_*\it Q_a^s = \SSS_{a}E \\
 \hskip1.5cm R^q\pi_*\it Q_a^s = 0\ \ \mathit{for}\ \ q>0,\hskip2cm 
\end{equation}
\label{pi}
where $\SSS_{a}$ is the Schur functor associated to the partition $a$.

We will use (\cite{semiample}, Lemma(3.3)) to reformulate 
the  following definition of Sommese,  see Prop. 1. 7 in \cite {Sommese}:

\begin{Def} \label{Def} A vector bundle $E$ is k ample if \\
1) E is semiample and\\
2)  Given any coherent sheaf $\mathcal{F}$ on $X$, there exists $N(\mathcal{F})\in\NN$  such that for $q>k$ and 
 any $n\geq N(\mathcal{F})$ one has
 $$ H^q(X,E^{\otimes n} \otimes \mathcal{F})=0. $$
\end{Def}

\smallskip

Here we consider the 

\bigskip 

{\bf{Conjecture}}: \label{con} The line bundle $\it Q_a^s$ on 
$\mathcal{F}l_s(E)$ is $k$-ample if and only if  the vector bundle 
$\SSS_{a}E$ on $X$  is $k$-ample.\\

In \cite{ample} we proved this conjecture  for $k=0$, that is for ample bundles.

We will prove a weak version of this conjecture.

\section{The proof}

\begin{thm}
 If  $\SSS_{a}E$ is   $k$-ample, then $\it Q_a^s$ is  $k$-ample.
\end{thm}

\begin{proof}
 Assume that  $\SSS_{a}E$ is $k$-ample. Thus by definition
 $\OOO_{\PP(\SSS_{a}E)}(1)$ is $k$-ample. The restriction of
 $\it Q_a^s$ to the fibers of $\mathcal {F}l_s(E)$ is very ample. Thus $\it Q_a^s$ yields an embedding generalizing the one of Plucker: 
 $$p:\mathcal {F}l_s(E) \hookrightarrow \PP(\pi_*\it Q_a^s) = \PP(\SSS_{a}E).$$
  The corresponding
 restriction of the $k$-ample line bundle $\OOO_{\PP(\SSS_{a}E)}(1)$ is $k$-ample, too, 
 and this is just $\it Q_a^s$ itself.
 \end{proof}

 \begin{rem}
The proof here is much shorter than  in \cite{ample} for $k=0$	
 \end{rem}

\begin{thm} \label{thm2} assume that
$\it Q^a_s$  is $k$-ample, then $\SSS_{a}E)$ satisfies  the condition 2) in Def  \ref{Def}:

 Given any coherent sheaf $\mathcal{F}$ on $X$, there exists $N(\mathcal{F})\in\NN$  such that for $q>k$ and 
 any $n\geq N(\mathcal{F})$ one has
 $$ H^q(X,(\SSS_{a}E)^{\otimes n} \otimes \mathcal{F})=0. $$
 \end {thm}

\begin{proof}

Let us  denote  $\mathcal {F}l_s(E)= \it Y.$ Let $\mu=lcd \{1,2,\ldots d\}$ and consider\\
  $\it Y_\mu= \underbrace{\it Y\times _X \ldots \times_X \it Y }_{\mu\ times}$\ \ 
$\pi_i :\it Y_\mu\longrightarrow Y $\ for $i=1,\ldots,\mu$,\ \ 
 $\pi :\it Y_\mu\longrightarrow X. $ 
 
 Since $\it Q^a_\mu=\pi^*_1 \it Q^a_s\otimes  \ldots \otimes \pi^*_\mu \it Q^a_s$ is  $k$-ample when restricted 
 to the fibers of $\pi$, it is $k$-ample by Proposition (1.8) of \cite{Sommese}.   
 
 According to Definition  \ref{Def} for any coherent sheave $\mathcal{F}$ on $X$ there exists $l_0(\mathcal{F}) $ 
 so that  for any $l\geq l_0(\mathcal{F})$ and any $q>k$
 $$ H^q(\it Y_\mu, \it (Q^a_\mu)^l\otimes \pi^*\mathcal{F})=0. $$
 
By the K\"unneth formula and (1.1) we get 
 $$ H^q(\it Y_\mu, \it (Q^a_\mu)^l\otimes \pi^*\mathcal{F}) 
 \simeq  H^q (X,(\SSS_{la}E)^{\otimes \mu}\otimes \mathcal{F}).$$
 
To prove Theorem \ref{thm2} it suffices to prove the following lemma.

\begin{lem}\label{inc}
 For any partition $a$ there is a finite set of partitions $\sigma(a)$ such that for any $n$
 each irreducible subfactor  $\SSS_bE$ of $(\SSS_aE)^{\otimes n}$ is isomorphic to a subfactor of
 $(\SSS_{la}E)^{\otimes \mu}\otimes \SSS_fE$, where $f\in\sigma(a)$ and $n|a|=\mu l|a|+|f|$. 
\end{lem}

Let
$$n_ 0(\mathcal{F})={\it max}_{f\in\sigma(a)}\{\mu l_0(\SSS_fE\otimes \mathcal{F})+[|f|/|a|]\},$$
where $[x]$ denotes the integral part of $x$.

Granted lemma \ref{inc} one finds for $n\geq n_ 0(\mathcal{F})$ and $q>k$ that
  $$ H^q(X, (\SSS_aE)^{\otimes n}\otimes \mathcal{F})=0. $$
 Thus $\SSS_{a}E$ satisfises the condtion 2) in \ref{Def}.

It remains to proof Lemma \ref{inc}.

Recall that the weight $|a|$ of a partition $a=(a_1,a_2,\ldots)$ is the sum of its terms,
$|a|=a_1+a_2+\ldots.$ The dominance partial ordering of partitions is defined in 
\cite{Fulton} by:\\
Let $a=(a_1,a_2,\ldots a_d) $ and $ b=(b_1,b_2,\ldots b_d)$ be two partitions.\\ 
Then \ \ $a \preceq b$ \ \ if
\begin{align}
\nonumber a_1&\leq b_1 \\
\nonumber a_1+a_2&\leq b_1+b_2\\
\nonumber \ldots&\leq\ldots\\
\nonumber a_1+\ldots +a_{d-1}&\leq b_1+\ldots +b_{d-1}.\\
\nonumber a_1+\ldots +a_{d-1}+a_d&= b_1+\ldots +b_{d-1}+b_d.
\end{align}

  We extend this definition to non-zero partitions of not necessarily equal weight:

 $$a \preceq b \ \ \ \ \  {\mbox{if} } \ \  \ \ \ |b|\ a \ \preceq |a| \ \ b.$$

In this terminology a well-known property of the\\ Littlewood-Richardson rules is
\begin{prop}\label{b}
For each irreducible subfactor  $\SSS_bE$ of $(\SSS_aE)^{\otimes n}$ one has $b\preceq a$. 
\end{prop}

 \begin{Not}\label{not3}
For finite sequences $b=(b_1,\ldots,b_s)$ and  $c=(c_1,\ldots c_t)$ we set  
$$b\vee c  = (b_1,\ldots,b_s,c_1,\ldots c_t).$$
For a given positive integer $d$ let 
$$\LLL(d)= \{(l_1,\ldots,l_r) \ | \ r,\ l_i\in \NN ,\  l_1+l_2+\ldots+l_r=d  \}. $$

If $L=(l_1,\ldots,l_r)\in \LLL(d)$ and $a \in\QQ_{\geq 0}^d,$ then for $1\leq i\leq r,$
$a(L,i)$ is a sequence
of length $l_i,$ such that $a=a(L,1)\vee a(L,2)\ldots\vee a(L,r).$

Let 
$${\bf{1}}_{i}=(\underbrace{1,1,\ldots, 1}_{i\ \ times} )$$ and
 $$v(L,a)=\frac{|a(L,1)|}{l_1} {\bf{1}}_{l_1}\vee \ldots \vee \frac{|a(L,r)|}{l_r} {\bf{1}}_{l_r}.$$
 
 Note that for partitions $a$ the sequences $v(L,\mu a)$ are partitions.
 
\end{Not} 

Let $Z(a)$ be the set of partitions $b$ for which $b\preceq a$. In \cite{ample} we proved
\begin{lem} \label{finite} (\cite{ample}, Lemma (3.11))
 There is a finite set $\sigma(a)$ of partitions  such that if
  $b\in Z(a)$ then $b$ can be written as  $$b=c + \displaystyle{\sum_{L\in \LLL(d)} m_L\  v(L,\mu a)}, \ \ 
  c\in\sigma(a)$$ and $m_L\in\ZZ_{\geq 0}$ \ for all $L\in \LLL(d)$.
\end{lem}

\begin{lem} \label{vinc} (\cite{ample}, Lemma (3.12))
 Let $a$ be a partition. For any $L\in \LLL(d)$,\ \ 
 $\SSS_{v(L,\mu a)}$ is a subfactor of $ (\SSS_aE)^{\otimes \mu}.$
\end{lem}

The Littlewood-Richardson rules have a well-known semigroup property:
\begin{prop}\label{ZZ}
Let  $a,b,c,d, e, f$  be partitions of length $d$. If \\
 $c  \in  a\otimes b $ and  $f  \in d \otimes e,  $
then \ \ $  (c+f)  \in (a+d) \otimes ( b+e ) .$
 \end{prop}
 
 Together with lemma \ref{vinc} this property implies
 \begin{prop}\label{g}
  Let $g= \displaystyle{\sum_{L\in \LLL(d)} m_L\  v(L,\mu a)}$ and $M=\displaystyle{\sum_{L\in \LLL(d)} m_L}.$
Then $\SSS_gE$ is a subfactor of $(\SSS_{Ma}E)^{\otimes \mu}.$
 \end{prop}
 
 Propositions \ref{finite} and \ref{g} imply lemma \ref{inc}.
Thus theorem \ref{thm2} is proven, too.
\end{proof}

To prove the conjecture\ref{con},  we need to prove the semiampleness of $\SSS_E$ when  $\it Q_a^s$ is $k$-ample,  wich is missing here.

  \end{document}